\documentclass[a4paper]{amsart}

\usepackage{latexsym, amssymb,amsmath, amsthm,amsfonts}

\newcommand{\kk}{\Bbbk}

\def\SL{\operatorname{SL}}

\def\SL2{\operatorname{SL}_{2}(K)}

\def\GL2{\operatorname{GL}_{2}(K)}

\def\INVSL2{$K[V]^{operatorname{SL}_{2}(K)}$}
\def\INVSO2{$K[V]^{operatorname{SO}_{2}(K)}$}
\def\INVGL2{$K[V]^{operatorname{GL}_{2}(K)}$}

\def\depth{\operatorname{depth}}

\def\GL{\operatorname{GL}}
\def\SL{\operatorname{SL}}

\def\Z{\mathbb{Z}}

\def\Tr{\operatorname{Tr}}

\newtheorem{Lemma}{Lemma}

\newtheorem{cor}[Lemma]{Corollary}
\newtheorem{prop}[Lemma]{Proposition}

\newtheorem*{Corollary of Conjecture}{Corollary of Conjecture}

\theoremstyle{definition}

\theoremstyle{remark}

\newtheoremstyle{Acknowledgments}
  {}
    {}
     {}
     {}
    {\bfseries}
    {}
     {.5em}
     {\thmname{#1}\thmnumber{ }\thmnote{ (#3)}}
\theoremstyle{Acknowledgments}

\title[Depth of quotients of invariant rings]{On the depth of quotients of modular invariant rings by transfer ideals}

\author{Jonathan Elmer}
\address{Middlesex University\\
The Burroughs, Hendon, London\\
NW4 4BT UK}
\email{j.elmer@mdx.ac.uk}

\author{M\"{u}fit Sezer}
\address{Bilkent University, Department of Mathematics\\
Cankaya, Ankara \\06800 Turkey } \email{sezer@fen.bilkent.edu.tr}

\thanks{The second author is supported by a grant from T\"UBITAK:115F186 }
\date{\today}

\begin{document}
\maketitle
\begin{abstract}
Let $G$ be a finite group, and $V$ a finite dimensional vector space over a field $\kk$ of characteristic dividing the order of $G$. Let $H \leq G$. The transfer map $\Tr_H^G: \kk[V]^H \rightarrow \kk[V]^G$ is an important feature of modular invariant theory. Its image is called a transfer ideal $I^G_H$ of $\kk[V]^G$, and this ideal, along with the quotients $\kk[V]^G/I^G_H$ are widely studied.

In this article we study $\kk[V]^G/I$, where $I$ is any sum of transfer ideals. Our main result gives an explicit regular sequence of length $\dim(V^G)$ in $\kk[V]^G/I$ when $G$ is a $p$-group. We identify situations where this is sufficient to compute the depth of $\kk[V]^G/I$, in particular recovering a result of Totaro. We also study the cases where $G$ is cyclic or isomorphic to the Klein 4 group in greater detail. In particular we use our results to compute the depth of $\kk[V]^G/I^G_{\{1\}}$ for an arbitrary indecomposable representation of the Klein 4 group.

\end{abstract}
\section{Introduction}

Let $\kk$ be an infinite field and $V$ a
finite-dimensional  $\kk$-vector space, and $G \leq \GL(V)$ a
finite group. Then the induced action  on $V^*$  extends to the
symmetric algebra $\kk[V]:=S(V^*)$ by the formula $$\sigma
(f)=f\circ \sigma^{-1}$$ for $\sigma \in G$ and $f\in \kk[V]$. The
algebra of fixed points $\kk[V]^G$ is called the \emph{ring of
invariants}, and is the central object of study in invariant
theory. 

Now let $H$ be a subgroup of $G$. There is a well-defined $\kk$-linear map
\[\Tr^G_H: \kk[V]^H \rightarrow \kk[V]^G\]
\[f \mapsto \sum_{\sigma \in G/H} \sigma f\] which is called the {\it relative transfer} from $H$ to $G$. If $H$ is trivial we call this the {\it transfer} to $G$ and denote it by $\Tr^G$. 

If $[G:H]$ is nonzero in $\kk$, then it is easy to show that $\Tr^G_H$ is surjective. In particular, $\Tr^G$ is a surjective map if $|G|$ is not divisible by the characteristic of $\kk$.

From now on suppose $\kk$ is a field of positive characteristic $p$ dividing $|G|$. Then if $H$ is a $p$-subgroup of $G$, but not a Sylow-$p$-subgroup, the transfer $\Tr^G_H$ is not surjective and its image is a proper ideal of $\kk[V]^G$. We denote this ideal by $I^G_H$. More generally, to any set $\mathcal{X}$ of subgroups of $G$ we associate the ideal

\begin{equation}\label{transferideal}
I^G_{\mathcal{X}} = \sum_{X \in \mathcal{X}} \Tr^G_X(\kk[V]^X). 
\end{equation} 

Transfer maps and their images are an important feature of modular invariant theory. For instance, a long-standing conjecture of Shank and Wehlau \cite{ShankWehlauTransfer} states that if $G$ is a $p$-group then $\kk[V]^G$ is polynomial if and only if $I^G_{\{1\}}$ is a principal ideal. Other investigations concern the transfer ideal $I:= I^G_{<P}$ where $P$ is a Sylow-$p$-subgroup and $<P$ denotes the set of all proper subgroups of $P$. This can be shown to be independent of the choice of Sylow-$p$-subgroup. Notable are a conjecture of Wehlau et al \cite{QueensReport} stating that $\kk[V]^G/I$ is always generated by invariants of degree $ \leq |G|$, and a recent result of Totaro \cite{Totaro} stating that $\kk[V]^G/I$ is always a Cohen-Macaulay ring. Both of these conform to a general philosophy that $\kk[V]^G/I$ behaves rather like a non-modular ring of invariants.

In the present article we consider quotient rings of the form $R^G_{\mathcal{X}}:= \kk[V]^G/I^G_{\mathcal{X}}$ where $G$ is a $p$-group and $\mathcal{X}$ is any set of proper subgroups of $G$. Our main result (Proposition \ref{regularsequence}) gives an explicit regular sequence of length $\dim(V^G)$ in $R^G_{\mathcal{X}}$. In particular, we recover Totaro's result in the special case of a $p$-group by elementary means. We also give some new results on the depth of $R^G_{\mathcal{X}}$ where $G$ is a cyclic $p$-group, and classify indecomposable representations $V$ of the Klein 4-group $G$ for which $\kk[V]^G/I_{\{1\}}^G$ is Cohen-Macaulay.

 \section{A regular sequence} 
In this section let $\kk$ be a field and $A = \oplus_{i \geq 0}A_i$ a commutative graded $\kk$-algebra with $A_0=\kk$. Let $M$ be a graded $A$-module. An $M$-{\it regular element} is an element $a \in A$ such that the map $\phi_a: M \rightarrow M$ defined by
\[\phi_a(m) = am\] is injective. An $M$-{\it regular sequence} is a sequence $a_1,a_2, \ldots, a_r$ such that $a_1$ is $M$-regular and for all $i=2, \ldots, r$, $a_i$ is $M/(a_1, \ldots, a_{i-1})M$-regular.

An $M$-regular sequence is called {\it maximal} if it cannot be extended to a longer $M$-regular sequence. If $A$ is Noetherian and $M$ is finite, then it can be shown that all maximal $M$-regular sequences have the same length. Define $A_+  = \oplus_{i \geq 1}A_i$.  Then we call the length of a maximal $M$-regular sequence in  $A_+$ the {\it depth} of $M$. The depth of $A$ is the depth of $A$ as a module over itself.  

For further results on depth and regular sequences we recommend \cite{BrunsHerzog}.

Now let $G$ be a $p$-group and adopt the notation introduced in the introduction. Let $v_1, \ldots, v_l$ be a basis of the fixed-point space $V^G$ and extend to a basis $v_1, \ldots, v_n$ of $V$. Let $x_1, x_2, \ldots, x_n$ be the corresponding dual basis of $V^*$. Then $\kk[V]$ is the polynomial ring generated by $x_1, x_2, \ldots, x_n$. For $i=1, \ldots, l$ we define
\begin{equation}\label{norms}
N_i = N^G(x_i) = \prod_{\sigma \in G} \sigma x_i.
\end{equation}

Notice that $N_i$ is a monic polynomial of  degree $|G|$ as a polynomial in $x_i$, and degree zero as a polynomial in $x_j$ for any $j=1, \ldots, i-1,i+1, \ldots, l$. To see the second statement, note that 
\[N_i(v_j) = \prod_{\sigma \in G} (\sigma x_i)v_j = \prod_{\sigma \in G} x_i(\sigma^{-1}v_j) = \prod_{\sigma \in G} x_i(v_j) = 0.\]
Further, let $\mathcal{X}$ be a set of proper subgroups of $G$. Given $f \in \kk[V]^G$, we denote by $\overline{f}$ the image of $f$ in $R^G_{\mathcal{X}}$. With this notation we have the following which was proven for a cyclic group of prime order in \cite[Theorem~12]{MR3437279}.

\begin{prop}\label{regularsequence}
\[\overline{N_1}, \ldots, \overline{N_l}\] is a regular sequence in $R^G_{\mathcal{X}}$.
\end{prop}

\begin{proof} We first show that $\overline{N_i}$ is regular on $R^G_{\mathcal{X}}$ for any $i=1, \ldots, l$. Let $i \in \{1, \ldots, l\}$ and suppose there exists $f \in \kk[V]^G$ such that $\overline{f}\overline{N_i} = 0$ in $R^G_{\mathcal{X}}$. In turn this means for each $X \in \mathcal{X}$, there exists a polynomial $g_X$ such that
\begin{equation}
fN_i = \sum_{X \in \mathcal{X}} \Tr^G_X(g_X).
\end{equation}

We view $g_X$ as a polynomial in  $x_i$ with coefficients in $\kk[x_1, \ldots, x_{i-1},x_{i+1},x_n]$.   Since $N_i$ is a monic polynomial in $x_i$, we may divide $g_X$ by $N_i$, finding unique polynomials $q_X$ and $r_X$ such that
\[g_X = q_X N_i + r_X\] where the $x_i$-degree of $r_X$ is $<|G|$.  As $v_i \in V^G$, the action of $G$ on $\kk[V]$ does not increase $x_i$-degrees so the uniqueness of this decomposition implies that $q_X, r_X \in \kk[V]^X$ and so 
\[fN_i = \sum_{X \in \mathcal{X}} \Tr_X^G(q_XN_i + r_X) = \sum_{X \in \mathcal{X}} N_i \Tr_X^G(q_X)+\Tr^G_X(r_X).\]
This can be rearranged as
\[0 = (f-\sum_{X \in \mathcal{X}}\Tr^G_X(q_X))N_i + \sum_{X \in \mathcal{X}}\Tr^G_X(r_X).\]
Since $\deg_{x_i}\Tr^G_X(r_X) < |G|$ for each $X$, comparing the $x_i$-degrees in the above equation gives  
\[f-\sum_{X \in \mathcal{X}}\Tr^G_X(q_X)) = 0,\]
i.e. $f \in I^G_{\mathcal{X}}$, which shows that $\overline{f} =0$ as desired. Therefore, $\overline{N_i}$ is regular on $R^G_{\mathcal{X}}$.

Now assume $1<j \leq l$ and $\overline{N_1}, \ldots, \overline{N_{j-1}}$ is a regular sequence in $R^G_{\mathcal{X}}$. Suppose $\overline{N_j}$ is not regular on $R^G_{\mathcal{X}}/(\overline{N_1}, \ldots, \overline{N_{j-1}})R^G_{\mathcal{X}}$. This means there exist $f_1, \ldots,f_j \in \kk[V]^G$ such that 
\[\overline{f_1}\overline{N_1} + \ldots + \overline{f_j}\overline{N_j} = 0\]
in $R^G_{\mathcal{X}}$. This in turn means there exists, for each  $X \in \mathcal{X}$, an element $g_X \in \kk[V]^X$ such that

\begin{equation}
f_1N_1+ \ldots+ f_jN_j = \sum_{X \in \mathcal{X}} \Tr^G_X(g_X).
\end{equation}

We now divide each $f_i$ ( for $i=2 \ldots j)$ by $N_1$ with remainder, writing
\[f_i = q_{i,1}N_1+r_{i,1}.\]
 We obtain for the right-hand side 
\[f_1N_1+ ((q_{2,1})N_1+r_{2,1})N_2 + ((q_{3,1})N_1+r_{3,1})N_3 + \ldots + ((q_{j,1})N_1+r_{j,1})N_j.\]
By collecting together all terms divisible by $N_1$ to replace $f_1$, and replacing each $f_i$ by $r_{i,1}$, we can assume $f_i$ has degree $< |G|$ as a polynomial in $x_1$. Continuing this process, dividing by each $N_i$ $(i=2, \ldots, l)$ in turn and relabeling if necessary, allows us to assume $f_i$ has degree $<|G|$ in all variables $x_1, \ldots, x_{i-1}$. Further, for each $X \in \mathcal{X}$ we write
 \[g_X = q_{X,1} N_1 + q_{X,2}N_2 + \ldots q_{X,j}N_j + r_X\] where the degree of $q_{X,i}$ is $<|G|$ in each variable $x_1, \ldots, x_{i-1}$ and the degree of $r_X$ is $<|G|$ in each of $x_1, \ldots, x_j$. With further relabeling (bringing terms $\Tr^G_X(q_{X,i})N_i$ to the left hand side) we obtain the expression 

\begin{equation}
f_1N_1+ \ldots+ f_jN_j = \sum_{X \in \mathcal{X}} \Tr^G_X(r_X)
\end{equation}
in which each $f_i$ has degree $<|G|$ as a polynomial in each variable $x_1, \ldots, x_{i-1}$ and each $r_X$ has degree $< |G|$ as a polynomial in each variable $x_1, \ldots, x_j$. 

Now notice that that, as a polynomial in $x_1$, the term $f_1N_1$ has degree $\geq |G|$ unless $f_1=0$. All other terms in the expression have $x_1$-degree $< |G|$. This shows that $f_1=0$. Considering degrees in $x_2, \ldots, x_{j-1}$ now shows that $f_2 = f_3 = \ldots =f_{j-1} = 0$, too. So we're left with 
\[f_jN_j = \sum_{X \in \mathcal{X}} \Tr^G_X(r_X).\]
Since we already know that $N_j$ is regular on $R^G_{\mathcal{X}}$, we must have $f_j = 0$. This completes the proof.

\end{proof}

\section{Upper Bounds}
Proposition \ref{regularsequence} implies that 
\[\depth(R^G_{\mathcal{X}}) \geq \dim(V^G)\]
for any set $\mathcal{X}$ of proper subgroups of the $p$-group $G$. In order to compute the exact depth of $R^G_{\mathcal{X}}$, we need an upper bound for depth. It is well-known that, for any Noetherian $\kk$-algebra $A$, $\depth(A) \leq \dim(A)$, see e.g. \cite[Proposition~1.2.12]{BrunsHerzog}. With this in mind, we want to compute $\dim(R^G_{\mathcal{X}})$. 

As a first step, we note that $I^G_{\mathcal{X}} = I^G_{\overline{\mathcal{X}}}$, where $\overline{\mathcal{X}}$ is obtained from $\mathcal{X}$ by including all subgroups of every $X \in \mathcal{X}$ and their $G$-conjugates. This follows easily from e.g. \cite[Equation~(1)]{FleischmannTransfer}. Now we have

\begin{prop}
\label{dim}
\[\dim(R^G_{\mathcal{X}}) = \max \{\dim(V^Q): Q \not \in \overline{\mathcal{X}}\}.\]
\end{prop}

\begin{proof} Note that $\kk[V]^G$ can be interpreted as the ring of regular functions on the quotient $V/G$. Then by definition we have 
\[\dim(R^G_{\mathcal{X}}) = \dim(\mathcal{V}(I^G_{\mathcal{X}}))\] where, for an ideal $J$ of $\kk[V]^G$, $\mathcal{V}(J)$ denotes the set of $G$-orbits vanishing on $J$. By \cite[Proposition~12.4(iii)]{FleischmannTransfer} we have
\[\mathcal{V}(I^G_{\mathcal{X}}) = i_G^*(\bigcup_{Q \le G: Q \not \in \overline{\mathcal{X}}} V^Q)\] where $i_G$ is the inclusion map $\kk[V]^G \hookrightarrow \kk[V]$ and $i_G^*: V \rightarrow V/G$ is its dual. As the dimension of an algebraic variety is the largest dimension of an irreducible component, the result follows.
\end{proof}

Now in case $\mathcal{X} = \{X: X < G\}$ we get $$\dim(V^G) \leq \depth(R^G_{\mathcal{X}}) \leq  \dim(R^G_{\mathcal{X}}) = \dim(V^G).$$ Consequently we obtain

\begin{cor}[Totaro]\label{totaro}
$R^G_{<G}$ is a Cohen-Macaulay ring.
\end{cor}

For stronger upper bounds, we use the following, which is inspired by \cite[Theorem~3.17]{FleischmannCohomology}.

\begin{prop}\label{ub} Let $\mathcal{X}$ be a family of subgroups of $G$. Let $K\le H$ for all $H \in \mathcal{X}$. If $((\bigcap_{H \in \mathcal{X}} I_K^H)\setminus I_K^G)\cap \kk[V]^G\neq \emptyset$, then we have $$\depth(R^G_K)\le \dim(R^G_\mathcal{X}).$$
\end{prop}
\begin{proof} Let $H \in \mathcal{X}$ and let $f\in (I_K^H\setminus I_K^G)\cap \kk[V]^G$. Then $f=\Tr^H_K(g)$ for some $g\in \kk[V]^K$. For $h \in\kk[V]^H$  we have $$\Tr^G_K(gh)=\Tr^G_H(\Tr^H_K(gh))=\Tr^G_H(h\Tr^H_K(g))=\Tr^G_H(hf)=f\Tr^G_H(h).$$ This  shows that $I^G_H$ annihilates $f+I^G_K\in \kk[V]^G/I_K^G$. It follows that $I^G_{\mathcal{X}}$ annihilates $f+I^G_K\in \kk[V]^G/I_K^G$.
Therefore $I^G_{\mathcal{X}}$ is contained in one of the associated primes of $\kk[V]^G/I_K^G$ in $\kk[V]^G$. Since the depth of a module is smaller than the coheights of its associated primes, see \cite[Proposition~1.2.13]{BrunsHerzog}, the result follows.
\end{proof}
\section{Cyclic groups}
In this section $G$ denotes a cyclic group of order $p^r$. There are exactly $p^r$ indecomposable $\kk G$-modules $V_1, \dots , V_{p^r}$ and each indecomposable module $V_n$ is afforded by a Jordan block of size $n$.  We fix a generator $\sigma$ of $G$.
We choose a basis $e_1, e_2, \dots , e_n$ for $V_n$ such that the action of $\sigma$  is given by $\sigma (e_i)=e_i+e_{i-1}$ for $2\le i\le n$ and $\sigma (e_1)=e_1$.  Define $\Delta=\sigma-1$. We consider the subgroups $K\subseteq H$
of $G$ generated by $\sigma^{p^a}$ and $\sigma^{p^b}$, respectively. We identify a case when the premise of  Proposition \ref{ub} is attained. 

\begin{Lemma}

Assume that    $\kk [V]$ contains a summand which is  isomorphic to $V_n$ with  $p^a-p^b+1 \le n\le p^a-1$. Then $(I^H_K\setminus I^G_K)\cap \kk [V]^G \neq \emptyset$.
\end{Lemma}
\begin{proof}
Clearly, $e_1\in V_n^G$. We demonstrate that $e_1\in (I^H_K\setminus I^G_K)$. Note that $\Tr^H_K$ applies to $K$-invariants and a basis element $e_i$ is in $V_n^K$ if and only if  $\sigma^{p^a}(e_i)=e_i$, or equivalently $\Delta^{p^a}(e_i)=0$. It follows that $V_n^K=V_n$ since $n\le p^a-1$.  
The set $1, \sigma^{p^b}, \sigma^{2p^b}, \dots ,\sigma^{p^a-p^b}$ is a complete set of representatives for $H/K$. Therefore we have
\begin{alignat*}{2} \Tr^H_K=&(1+\sigma^{p^b}+\sigma^{2p^b}+ \cdots +\sigma^{p^a-p^b})\\
=&(1+\sigma+\sigma^2+ \cdots +\sigma^{p^{a-b}-1})^{p^b}\\
=&(\Delta^{p^{a-b}-1})^{p^b}=\Delta^{p^a-p^b}.
\end{alignat*}
 Then by assumption on $n$, we get $e_1\in I^H_K $ because $\Tr^H_K(e_{p^a-p^b+1})=e_1$ and $e_{p^a-p^b+1}\in V_n^K$. Similarly, $1, \sigma, \dots , \sigma^{p^a-1}$ is a complete set of representatives for $G/K$ and so  $\Tr^G_K=(1+\sigma+\sigma^2+ \cdots +\sigma^{p^a-1})=\Delta^{p^a-1}$. It follows that $e_1\notin I^G_K$ because $\Delta^{p^a-1}(V_n)=0$ by assumption on $n$.
\end{proof}
 
For subgroups  $K\subseteq H$ in $G$ we have the decomposition  $\Tr^G_{K}=\Tr^G_{H}  \circ\Tr^{H}_{K}$, so $I^G_{K}\subseteq I^G_{H}$. It follows that for a collection
$\mathcal{X}$ of subgroups, we have $I^G_{\mathcal{X}}=I^G_{H}$, where $H$ is the largest subgroup in $\mathcal{X}$. So we consider just subgroups of $G$, rather than collections of subgroups. 

\begin{prop}
Let $K$ be a subgroup of $G$, generated by $\sigma^{p^a}$. If there is $0< b \le a$ such that there is an indecomposable summand $V_n$ of $\kk [V]$ with 
$p^a-p^b+1 \le n\le p^a-1$, then $$\depth(R^G_K)\le \dim(V^{\sigma^{p^{b-1}}}) \le p^{b-1}\dim(V^G).$$
\end{prop}
\begin{proof}
Let $H$ denote the subgroup of $G$ generated by $\sigma^{p^{b}}$. Then from the  previous lemma we have that $(I^H_K\setminus I^G_K)\cap \kk [V]^G   \neq \emptyset$. So Proposition \ref{ub} applies for  $\mathcal{X}=\{H\}$ and we get that 
$\depth(R^G_K)\le \dim (R^G/I^G_H)$. On the other hand $\dim (R^G/I^G_H)=\max \{\dim(V^Q): H\subsetneqq Q \}$ by Proposition \ref{dim}. Since $G$ is cyclic, we have that $\max \{\dim(V^Q): H\subsetneqq Q \}=\dim(V^{\sigma^{p^{b-1}}})$ as desired. For the second inequality, note that for each summand $V_m$ of $V$ with basis $e_1, e_2, \dots ,e_m$, $V_m^G$ is spanned by $e_1$ and $V_m^{\sigma^{p^{b-1}}}$ is spanned by $e_1, \dots , e_k$, where $k=\min (p^{b-1}, m)$.
\end{proof}
 
We note an application for cyclic $2$-groups.
\begin{cor}
Let $G$ be a cyclic $2$-group of order $2^r$ and let $V=V_t$ be the indecomposable module of dimension $t\le 2^r$, where $t$ is odd. Then $\depth (R^G_{\{1\}})=\dim (V^G)$.   
\end{cor}
\begin{proof}
Note that the $d$-th symmetric power $S^d(V_t)$ is isomorphic to $\Omega^{-d} (\wedge^d(V_{2^r-t}))$ modulo induced modules, by \cite[Corollary~3.12]{SymondsCyclic}. Putting $d=2^r-t$ we get  $\wedge^d(V_{2^r-t}) = V_1$ and as $d$ is odd we have $\Omega^{-d}(V_1) = V_{2^r-1}$. Therefore, $V_{2^r-1}$ is a summand of the $d$-th homogeneous component of $\kk [V_t]$. So applying the previous proposition for $b=1$ and trivial $K$ yields  $\depth (R^G_{\{1\}})\le \dim (V^G)$. The reverse inequality follows from Proposition \ref{regularsequence}.
\end{proof}
\section{The Klein 4-group}

In this section we consider the action of the group $G = \langle \sigma_1,\sigma_2 \rangle \cong C_2 \times C_2$ on an indecomposable $\kk G$-module $V$, where $\kk$ is a field of characteristic 2. For each $V$, the depth of $\kk[V]^G$ was computed in \cite{ElmerFleischmann}. Further, if $\mathcal{X}$ contains all three proper subgroups of $G$, $R^G_{\mathcal{X}}$ is Cohen-Macaulay by Corollary \ref{totaro}. We will use the results of previous sections to compute $\depth(R^G_{\{1\}})$, for each indecomposable representation $V$ of $G$. Note that the results of Section 3 imply
\begin{equation}\label{ineq}
    \dim(V^G) \leq \depth(R^G_{\{1\}}) \leq \dim(R^G_{\{1\}})=\max_{1<Q<G} \dim(V^Q)
\end{equation}
and $R^G_{\{1\}}$ is Cohen-Macaulay if and only if the second inequality is an equality.

The classification of indecomposable modules for this group is given in \cite{Benson1}, and it is from this source we obtain our notation. To summarize, if $V$ is an indecomposable $\kk G$-module then $V$ is isomorphic to a module in the following list:

\begin{itemize}
    \item An even-dimensional representation of dimension $n$, denoted $V_{n, \lambda}$, where $\lambda \in \kk \cup \{\infty\}$;
    \item The representations $\Omega^m(\kk)$ where $m \in \Z$ and $\Omega$ represents the Heller shift operator. This representation has dimension $n = 2|m|+1$.
    \item The regular representation $\kk G$;
\end{itemize}

We analyse each case separately. 

\subsection{Even dimensions, not projective}
First suppose $V \cong V_{n, \lambda}$ where $\lambda \in \kk$. Write $n=2m$. Then we can choose a basis $\{v_1,v_2, \ldots, v_m, w_1, \ldots, w_m\}$ with respect to which the action of $G$ is as follows:

\[\sigma_i(v_j)=v_j \ \qquad i=1,2\]
\[\sigma_1(w_i) = w_i+v_i\]
\[\sigma_2(w_i) = w_i+\lambda v_i+ v_{i-1}\]
where we set $v_0=0$. $V_{n, \infty}$ is obtained from $V_{n,0}$ by swapping the roles of $\sigma_1$ and $\sigma_2$. Note that, since the three modules $V_{n,0}$, $V_{n,1}$ and $V_{n,\infty}$ are linked by a group automorphism, they have isomorphic invariant rings, and in particular $\depth(R^G_{\{1\}})$ is the same for all three modules. So from now on we assume $\lambda \in \kk$ and $\lambda$ is invertible.

Now it is  clear that $\dim(V^G) = m$. Moreover, if $\lambda \neq 1$ then $\dim(V^Q)=m$ for each maximal subgroup of $G$. So we obtain immediately from Equation \ref{ineq} that $\depth(R^G_{\{1\}}) = m$ and this quotient ring is Cohen-Macaulay.

In case $\lambda=1$ we have $\dim(V^{\sigma_1\sigma_2}) =m+1$. So $\dim(R^G_{\{1\}}) =m+1$. We claim that $\depth(R^G_{\{1\}}) = m$, so this quotient is not Cohen-Macaulay. Note that $V_{2,1}$ is not faithful, so we may assume $n \geq 4$.

To see this, consider the basis $\{x_1, \ldots, x_m,y_1, \ldots, y_m\}$ of $V^*$, dual to the given basis of $V$. The action of $G$ on this basis is given by

\[\sigma_i(y_j)=y_j \ \qquad i=1,2\]
\[\sigma_1(x_i) = x_i+y_i\]
\[\sigma_2(x_i) = x_i+y_i+y_{i+1}\]
where $y_{m+1}=0$.

Now we have $$y_m = \Tr^{\langle\sigma_1\sigma_2\rangle}(x_{m-1}) \in I^{\langle\sigma_1\sigma_2\rangle}_{\{1\}}.$$
Moreover $y_m \in \kk[V]^G$ and $y_m \not \in I^G_{\{1\}}$ because $V$ and hence $V^*$ is not projective. Applying Proposition
\ref{ub} with $\mathcal{X} = \{\langle\sigma_1\sigma_2\rangle\} $ and $K$ trivial brings
\[\depth(R^G_{\{1\}}) \leq \dim(R^G_{\langle \sigma_1\sigma_2 \rangle}) = m.\]
Combining this with Equation \ref{ineq} proves the claim.

\subsection{Odd dimensions - negative Heller shift.}

Suppose $V$ is isomorphic to $\Omega^m(\kk)$ where $m<0$ (for $m=0$ this representation is not faithful). Then
we can choose a basis $\{v_1,v_2, \ldots, v_{m+1}, w_1, \ldots, w_m\}$ with respect to which the action of $G$ is as follows:

\[\sigma_i(v_j)=v_j \ \qquad i=1,2\]
\[\sigma_1(w_i) = w_i+v_i\]
\[\sigma_2(w_i) = w_i+v_{i+1}\]

Now its clear that $\dim(V^G) = m+1$. Moreover, $\dim(V^Q)=m+1$ for each maximal subgroup of $G$. So we obtain immediately from Equation \ref{ineq} that $\depth(R^G_{\{1\}}) = m+1$ and this quotient ring is Cohen-Macaulay.

\subsection{Odd dimensions - positive Heller shift.}

Suppose $V$ is isomorphic to $\Omega^m(\kk)$ where $m>0$. Then
we can choose a basis $\{v_1,v_2, \ldots, v_{m}, w_1, \ldots, w_{m+1}\}$ with respect to which the action of $G$ is as follows:

\[\sigma_i(v_j)=v_j \ \qquad i=1,2\]
\[\sigma_1(w_i) = w_i+v_i\]
\[\sigma_2(w_i) = w_i+v_{i-1}\] where we set $v_{m+1}=v_0=0$.

Now it's clear that $\dim(V^G) = m$. Further, $\dim(V^{\sigma_1\sigma_2})= m$ but $\dim(V^{\sigma_i}) = m+1$ for $i=1,2$. So $\dim(R^G_{\{1\}}) = m+1$. 

Assume $m \geq 2$. We claim that $\depth(R^G_{\{1\}})=m$ and this quotient ring is not Cohen-Macaulay.

To see this, consider the basis $\{x_1, \ldots, x_m,y_1, \ldots, y_{m+1}\}$ of $V^*$, dual to the given basis of $V$. The action of $G$ on this basis is given by

\[\sigma_i(y_j)=y_j \ \qquad i=1,2\]
\[\sigma_1(x_i) = x_i+y_i\]
\[\sigma_2(x_i) = x_i+y_{i+1}.\]

 Then we have
\[\Tr^{\langle \sigma_1 \rangle}(x_m) = \Tr^{\langle \sigma_2 \rangle}(x_{m-1})=y_m.\] Moreover $y_m \in \kk[V]^G$ and $y_m \not \in I^G_{\{1\}}$ because $V$ and hence $V^*$ is not projective. Applying Proposition
\ref{ub} with $\mathcal{X} = \{\langle\sigma_1\rangle, \langle \sigma_2\rangle\} $ and $K$ trivial brings
\[\depth(R^G_{\{1\}}) \leq \dim(R^G_{\mathcal{X}}) = \dim(V^{\sigma_1\sigma_2})= m.\]
Combining this with Equation \ref{ineq} proves the claim.

For $m=1$ the ring of invariants is easy to compute: Let $N(x_1) = \prod_{g \in G}gx_1$. Note that $N(x_1)$ is $x_1^4$ modulo the ideal generated by $y_1, y_2$ in $\kk[V]$. So the common zero set of the invariants $y_1, y_2, N(x_1)$ contains only the point zero, so this set is a homogeneous system of parameters for $\kk[V]^G$. But since the product of their degrees is the group order, from a standard fact in invariant theory  we get that  
$\kk[V]^G = \kk[y_1,y_2,N(x_1)]$. In particular, $\kk[V]^G$ is a polynomial ring. Note that the ideal $\kk[V]_+^G \kk[V]$ is generated by $y_1, y_2, x_1^4$. Therefore $1, x_1, x_1^2, x_1^3$ forms a  $\kk $-basis for $\kk [V]/\kk[V]_+^G \kk[V]$.  Since $\Tr^G$ is a $\kk[V]^G$-linear map, it follows that $I^G_{\{1\}}$ is generated by $\Tr^G(1), \Tr^G(x_1), \Tr^G(x_1^2), \Tr^G(x_1^3)$. But $\Tr^G(1)=\Tr^G(x_1)= \Tr^G(x_1^2)=0$ and so $I^G_{\{1\}}$ is a principal ideal, generated by $\Tr^G(x_1^3) = y_1^2y_2+y_1y_2$. Clearly this is a regular element of $\kk[V]^G$.  Therefore by \cite[Theorem~2.1.3]{BrunsHerzog}, $R^G_{\{1\}}$ is Cohen-Macaulay.
 
\subsection{The regular representation}  
Let $V_R$ denote the regular  $\kk G$-module with a  basis $v_1, v_2, v_3,v_4$. The action of $\sigma_1, \sigma_2$ on $V_R$ is given by the permutations  $(1,2)(3,4)$ and $(1,3)(2,4)$, respectively. That is, we have $\sigma_i (v_j)=v_{\sigma_i (j)}$ for $i=1,2$ and $1\le j\le 4$. Note that  $\dim(V^{\sigma_1})=\dim(V^{\sigma_2})=\dim(V^{\sigma_1\sigma_2})=2$ and so by Proposition \ref{dim} we have $\dim(R^G_{\{1\}})=2$. We demonstrate that $\depth (R^G_{\{1\}})=2$ as well. Consider the  basis $x_1, x_2, x_3, x_4$ of $V_R^*$ dual to the given basis of $V_R$. Since the inverse transpose of a matrix representing a product of disjoint tranpositions is equal to itself, the action of $G$ on $V_R^*$ is again given by 
$$\sigma_i (x_j)=x_{\sigma_i (j)} \text { for } i=1,2 \text { and } 1\le j\le 4.$$
Since  $G$ permutes the variables,  $\kk [V_R]^G=\kk [x_1, x_2, x_3, x_4]^G$ is generated as a vector space by orbit sums $o(m)$ of monomials in $\kk [x_1, x_2, x_3, x_4]$. Note that $o(m)\in I^G_{\{1\}}$ for a monomial $m$ if the orbit sum contains four monomials. Otherwise the monomial $m$ has a non-trivial stabilizer and so $o(m)\notin I^G_{\{1\}}$. In this case the orbit sum $o(m)$ contains one or two monomials. We show that $N:=x_1x_2x_3x_4$ and $H:=o(x_1x_2)+o(x_1x_3)+o(x_1x_4)=x_1x_2+x_3x_4+x_1x_3+x_2x_4+x_1x_4+x_2x_3$ form a regular sequence in $R^G_{\{1\}}$. Assume that $N(\sum_t c_to(m_t))\in I^G_{\{1\}}$, where $c_t\in \kk $ and  $ o(m_t)\notin I^G_{\{1\}}$. So we may assume that each $o(m_t)$ contains one or two monomials. Note that the number of monomials in the orbits of $m_t$ and $Nm_t$ are the same and $N\sigma_i (m_t)=\sigma_i(Nm_t)$. Therefore $N(\sum_t c_to(m_t))=\sum_t c_to(Nm_t)$ which is a sum of orbits with one or two monomials. It follows that  $N(\sum_t c_to(m_t))\notin I^G_{\{1\}}$.

Next we show that $H$ is a non-zero divisor in the quotient ring  $R^G_{\{1\}}/NR^G_{\{1\}}=\kk [x_1, x_2, x_3, x_4]^G/(I^G_{\{1\}}, N)$. Note that the ideal $(I^G_{\{1\}}, N)$ is spanned by orbit sums $o(m)$ of monomials where $o(m)$ has 4 monomials or $m$ is divisible by $N$. Assume that $H(\sum_t c_to(m_t))\in (I^G_{\{1\}}, N)$ with $\sum_t c_to(m_t)\notin (I^G_{\{1\}}, N)$. Since $H$ is a homogeneous polynomial and  $(I^G_{\{1\}}, N)$ is a homogeneous ideal, we may assume $\sum_t c_to(m_t)$ is homogeneous as well. Each $m_t$ has two monomials in its orbit (if it has one, then $m_t$ is divisible by $N$), so each $m_t$  is fixed by one of the $\sigma_1, \sigma_2, \sigma_1\sigma_2$. It follows that $x_1, x_2, x_3, x_4$ group in two pairs and the  variables in the same pair appear with the same multiplicity in $m_t$. But $m_t$ is not divisible by $N$ so variables in one of the pairs do not appear at all in $m_t$.   Therefore we may assume that $m_t$ is equal to either $x_1^dx_2^d$, $x_1^dx_3^d$ or $x_1^dx_4^d$ for some positive integer $d$. So we have $H(c_1o(x_1^dx_2^d)+c_2o(x_1^dx_3^d)+c_3o(x_1^dx_4^d))\in (I^G_{\{1\}}, N)$. Note that if $c_1\neq 0$, then $x_1^{d+1}x_2^{d+1}$ appears in $Ho(x_1^dx_2^d)$. But  $x_1^{d+1}x_2^{d+1}$ does not appear in $Ho(x_1^dx_3^d)$ or in $Ho(x_1^dx_4^d)$. So  if $c_1\neq 0$, then $x_1^{d+1}x_2^{d+1}$ appears in 
  $H(c_1o(x_1^dx_2^d)+c_2o(x_1^dx_3^d)+c_3o(x_1^dx_4^d))$. But $x_1^{d+1}x_2^{d+1}$ is fixed by $\sigma_1$ and it so it has two monomials in its orbit and we get that it does not appear in a polynomial in $(I^G_{\{1\}}, N)$. This gives $c_1=0$. Along the same lines one sees that $c_2=0$ and $c_3=0$. 
  \subsection{Decomposable Representations}
   We let $V_1 \oplus V_2$ denote the direct sum of the
 $\kk G$-modules $V_1$ and $V_2$,  and let $kV$ denote the direct sum of $k$ copies of a $\kk G$-module $V$.  Assume that $V$ is isomorphic to $c_{n,\lambda}V_{n, \lambda}+d_m \Omega^m(\kk)$. Set $c=\sum c_{n, 1}$ and $d=\sum_{m>0}d_m$. We prove that the Cohen-Macaulay defect of $R^G_{\{1\}}$ is bounded above by $\max \{c, d\}$.
 \begin{prop}
 Assume the notation of the previous paragraph. Then we have $\dim (R^G_{\{1\}})- \depth (R^G_{\{1\}})\le \max \{c, d\}$. In particular, if $V$ is a direct sum of modules of the form $V_{n,\lambda}$ with $\lambda\neq 1$ and  $\Omega^m(\kk)$ with $m<0$, then $R^G_{\{1\}}$ is Cohen-Macaulay.
 \end{prop}
 \begin{proof}
 Note that the dimensions of the fixed point spaces
 of $\sigma_i$ and $G$ are equal in all summands except $\Omega^m(\kk)$ for $m>0$ and $i=1,2$. Furthermore $\dim (\Omega^m(\kk)^{\sigma_i})=\dim (\Omega^m(\kk)^{G})+1$ for  $m>0$ and $i=1,2$. It follows that $\dim (V^{\sigma_i})=\dim (V^G)+d$ for $i=1,2$. On the other hand the dimensions of the fixed point spaces
 of $\sigma_1\sigma_2$ and $G$ are equal in all summands except $V_{n,1}$, and $\dim (V_{n,1}^{\sigma_1 \sigma_2})=\dim (V_{n,1}^G)+1$. Therefore $\dim (V^{\sigma_1\sigma_2})=\dim (V^G)+c$. So from Proposition \ref{dim} we get $\dim (R^G_{\{1\}})=\dim (V^G)+\max \{c, d\}$. Since the depth of  $R^G_{\{1\}}$ is at least $\dim (V^G)$ by Proposition \ref{regularsequence}, the result follows.
\end{proof}  
\bibliographystyle{plain}
\bibliography{MyBib}

\end{document}